\documentclass[oneside]{amsart}

\usepackage{geometry}
\usepackage[utf8]{inputenc}		
\usepackage[foot]{amsaddr}		
\usepackage{amsthm}				
\usepackage{amsmath}			
\usepackage{amsfonts}			
\usepackage{amssymb}			
\usepackage{bbm}				

\usepackage{hyperref}			

\usepackage[
	backend=biber,
	style=alphabetic,
	maxnames=99,
	maxalphanames=5,
	doi=false,
	isbn=false
]{biblatex}
\renewbibmacro{in:}{}								
\DeclareFieldFormat{pages}{#1}
\DeclareFieldFormat{extraalpha}{#1}
\AtEveryBibitem{\clearfield{eprintclass}}

\DeclareLabelalphaTemplate{							
  \labelelement{
    \field[final]{shorthand}
    \field{label}
    \field[strwidth=3,strside=left,ifnames=1]{labelname}
    \field[strwidth=1,strside=left]{labelname}
  }
}

\numberwithin{equation}{section}

\newtheorem{theorem}{Theorem}[section]
\newtheorem{corollary}[theorem]{Corollary}
\newtheorem{lemma}[theorem]{Lemma}
\newtheorem{proposition}[theorem]{Proposition}

\theoremstyle{definition}

\newtheorem{example}[theorem]{Example}

\theoremstyle{remark}
\newtheorem*{remark}{Remark}

\theoremstyle{remark}

\renewcommand{\Re}{\operatorname{Re}}

\newcommand{\onebb}{\mathbbm 1}

\newcommand{\CC}{\mathbb C}								
\newcommand{\QQ}{\mathbb Q}
\newcommand{\RR}{\mathbb R}
\newcommand{\ZZ}{\mathbb Z}

\newcommand{\norm}[1]{\left\lVert #1\right\rVert}		
\newcommand{\abs}[1]{\left\lvert #1\right\rvert}
\newcommand{\divides}{\mid}

\DeclareMathOperator{\Mat}{Mat}
\DeclareMathOperator{\GL}{GL}

\DeclareMathOperator{\diag}{diag}						
\DeclareMathOperator{\tr}{tr}
\DeclareMathOperator{\rk}{rk}
\DeclareMathOperator{\pdet}{pdet}

\newcommand{\pmatrixtwo}[4]{\begin{pmatrix}#1&#2\\#3&#4\end{pmatrix}}

\begin{document}
\subjclass{Primary 11E20, 11P21. Secondary 33C10}
\keywords{Representations of quadratic forms, Gauss circle problem, matrix Bessel function, Poisson summation formula, Hankel transform}

\author{Gilles Felber}
\address{Alfréd Rényi Institute of Mathematics, Reáltanoda street 13-15, H-1053, Budapest}
\email{felber@renyi.hu}

\title{Counting representations of quadratic forms}
\begin{abstract}
We answer a question of C. S. Herz about the number of integral $k\times m$ matrices $T$ such that $T^tT\leq R$ for a fixed positive definite $m\times m$ matrix $R$. We give an asymptotic formula for the count. This can be seen as a generalization of the Gauss circle problem, as counting representation of quadratic forms by the identity $I_k$, and as counting integral points bounded by the Stiefel manifold $T^tT=R$. The main tool is a bound for Bessel functions of matrix argument that was proved over Jordan algebras.
\end{abstract}
\maketitle

\section{Introduction}
In 1955, Carl S. Herz asked in his seminal article \emph{Bessel Functions of Matrix Argument} \cite{Her55} about estimating the following set:
$$R_k(R):=\{T\in\Mat_{k,m}(\ZZ)\mid T^tT\leq R\}$$
for a fixed positive definite matrix $R\in\Mat_m(\RR)$. Here we used the Loewner order on symmetric matrices, defined by $A\leq B$ if and only if $B-A$ is positive semi-definite. Using the Poisson summation formula and the Hankel transform that he defined in the same article, Herz wrote an estimate for $R_k(R)$ given by the zero term of the dual side of the Poisson summation formula. Unfortunately, as he noted, the methods of his paper are not enough to prove an asymptotic formula. Decades later, Faraut--Travaglini \cite{FT87} and Kalliterakis \cite{Kal01} provided the necessary bounds on Bessel functions of matrix argument in the generalized setting of Jordan algebras. The goal of this paper is to answer the original question of Herz, and also to make analytic number theorists aware of the results over Jordan algebras, in the hope of more applications using these methods. Our main theorem is the following.

\begin{theorem}\label{thm R_k(R) asymptotic}
Let $k\geq m\geq1$ and $\epsilon>0$. Let $R\geq3$ be a positive definite matrix and $\lambda_{\min}(R)$ be its smallest eigenvalue. Then
\begin{align}\label{eq R_k(R) asymptotic}
R_k(R)=\frac{\Pi_m(0)\pi^{k/2}}{\Pi_m(k/2)}\det(R)^{k/2}+O_{k,m,\epsilon}\left(\det(R)^{k/2}\lambda_{\min}(R)^{-1+\frac m{k+2m+1}+\epsilon}\right).
\end{align}
Here $\Pi_m$ is a product of Gamma factors defined in Equation \eqref{eq definition Pi_m}.
\end{theorem}

This theorem gives a uniform asymptotic expansion of $R_k(R)$ for matrices $R\geq3$ with all their eigenvalues going to infinity. The restriction to $R\geq3$ is technical and can be slightly improved. The smallest eigenvalue of $R$ naturally arises in the proof when matrices of lower rank appear. The factor $\lambda_{\min}(R)^\epsilon$ only appears in the count of integral symmetric matrices of a given rank, and could be replaced by a $\log$ factor, or even removed with a more precise version of Theorem \ref{thm number of symmetric matrices}.

If $m=1$, we recover the classical case of the Gauss problem in dimension $k$ of counting integral vectors $x\in\ZZ^k$ such that $\norm x\leq r$, with $r^2=R\in\RR_{>0}$. In that case, the smallest eigenvalue is just $r^2$ and we get the error term $O(r^{k-2+\frac2{k+3}+\epsilon})$. This is very close to the best possible exponent $O(r^{k-2})$ for $k\geq4$. Matching upper and lower bounds for the error term are known in that case. It is easy to slightly improve our bound to the exponent $O(r^{k-2+\frac2{k+1}})$ using that Theorem \ref{thm number of symmetric matrices} is trivial in that case and
$$\#\{x\in\ZZ^k\mid\norm x^2=n\}=O(n^{k/2-1})$$
instead of Proposition \ref{pro bound r_k(M)}. In the case $m\geq2$, improved bounds on the number of representations of $M$ as $T^tT$ were proved by multiple people in specific cases (e.g. \cite{Rag59,Kit82a,DH13}), assuming at least information on the minimum of the quadratic form associated to $M$. These bounds could improve the final exponent in our estimate but require to split the dual sum of the Poisson summation formula in multiple cases.

For more details on the Gauss problem in dimension $k$, there are multiple surveys on the question, for example \cite{Sch04,IKKN06} . We also mention \cite{Krae00} that we consulted for classical proofs of the lower and upper bounds on the error term in the Gauss problem in $k$ dimensions, the monography \cite{Wal57}, and \cite{EL25} that we used for inspiration and as a comparison point.

The proof of Theorem \ref{thm R_k(R) asymptotic} is classical. Following the normalization of Herz, a matrix function $F:\Mat_{k,m}(\RR)\to\CC$ is \emph{radial} if there is a \emph{profile function} $f:\Mat_m(\RR)\to\CC$ such that $F(T)=f(\pi T^tT)$. We consider the characteristic function of the matrices $T\in\Mat_{k,m}(\RR)$ with $\pi T^tT\leq R$ and we convolute it with a radial bump function with shrinking support. Then we use the Poisson summation formula for $\ZZ^{km}$ to obtain the main term given by the term at 0 of the dual sum. It remains to bound the non-zero terms of the dual sum. For this, we use the Hankel transform developed by Herz. It is an integral against a generalized $J$-Bessel function of matrix argument define by
\begin{align}\label{eq definition A_k}
A_k^{(m)}(M)=\frac1{(2\pi i)^{m(m+1)/2}}\int_{\Re(Z)=X_0>0}e^{\tr(Z-MZ^{-1})}\det(Z)^{-k-(m+1)/2}dZ
\end{align}
for $M\geq0$ and a fixed matrix $X_0>0$. The Poisson summation and radial functions are discussed in Section \ref{sec Poisson summation}. In Section \ref{sec Bessel functions}, we introduced a bound on $A_k^{(m)}(M)$ that was computed on Jordan algebra using a stationary phase argument. Finally in Sections \ref{sec preliminaries} and \ref{sec proof of thm}, we proceed with the proof of the theorem, using bounds on the number of representations of $M\in\Mat_m(\ZZ)$ by $T^tT$ and a bound on the error given by the convolution with a bump function.

It would be of great interest to prove a lower bound for the error term in Equation \eqref{eq R_k(R) asymptotic}. We did not succeed in getting an interesting result. The main problem is to find $\delta>0$ as large as possible such that (assuming $R$ integral) there is no matrix $M\in\Mat_m(\ZZ)$ with
$$0\leq M\leq R+\delta I_m\quad \text{and}\quad M\not\leq R.$$
Comparing to the case $m=1$, we conjecture that the optimal error term in Equation \eqref{eq R_k(R) asymptotic} is $O_{k,m}(\det(R)^{k/2}\lambda_{\min}(R)^{-1})$ (at least for $k$ large compared to $m$). Proving $\delta\gg_{k,m}1$ would be sufficient to prove the corresponding lower bound. See \cite[Satz 5.9]{Krae00} for the details in the case $m=1$, which can be easily extended to our setting.

There are also many variations of the above problem that have been studied in the case $m=1$, such as considering the error on average or a shifted lattice. See \cite{EL25} for some examples. The primitive Gauss problem reduces the count to primitive vectors with the greatest common divisor of the entries equal to 1. We can naturally extend this condition by asking that the content of the matrices $T$ that we are counting is 1. We hope to return to these interesting questions in the future.

\subsection{Notation}
We (mostly) follow the notations of Herz \cite{Her55}. We usually consider $m$ by $m$ square matrices and $k$ by $m$ rectangular matrices with $k\geq m$. We write $p=\frac{m+1}2$ and $A[B]:=B^tAB$. We use the Loewner order defined by $A>B$ if $A-B$ is positive definite resp. $A\geq B$ if $A-B$ is positive semi-definite. We write $\Omega_m $ resp. $\bar\Omega_m$ for the set of $m$ by $m$ positive definite resp. positive semi-definite matrices. Note that $\Omega_m$ is a cone and $\bar\Omega_m$ is its closure.

We denote by $\norm T$ the operator norm of $T$ and $\norm T_\infty$ the maximum norm of $T$, that is
$$\norm T^2:=\lambda_{\max}(T^tT),\qquad\norm T_\infty:=\max_{i,j}\abs{t_{ij}}.$$
Recall that $\norm T_\infty\leq\norm T\leq\sqrt{km}\norm T_\infty$ for $T\in\Mat_{k,m}(\RR)$. The \emph{pseudo-determinant} of a matrix $M$, written $\pdet(M)$, is the product of its non-zero eigenvalues.

We consider the product of Lebesgue measures $dX=\prod_{i\leq j}dx_{ij}$ for a symmetric matrix $X$ and $dT=\prod_{i,j}dt_{ij}$ for a rectangular matrix $T$. We define the products of Gamma factors:
\begin{align}\label{eq definition Pi_m}
\Pi_m(\delta)&:=\pi^{m(m-1)/4}\prod_{i=1}^m\Gamma\left(\delta+\frac{i+1}2\right).
\end{align}
We denote the Fourier transform of $F:\Mat_{k,m}(\RR)\to\CC$ by
$$\hat F(S):=\int_{\Mat_{k,m}(\RR)}e(-S^tT)F(T)dT$$
where $e(Z):=e^{2\pi i\tr(Z)}$.

\subsection{Acknowledgment}
The author thanks Edgar Assing and Árpád Tóth for the enlightening discussions and advice. The research towards this paper was supported by the MTA–RI Lendület “Momentum” Analytic Number Theory and Representation Theory Research Group.

\section{The Poisson Summation Formula and Bessel functions of matrix argument}

\subsection{The Poisson formula}\label{sec Poisson summation}
In this section, we consider the Poisson summation formula for radial functions on $\Mat_{k,m}(\RR)$. First, let $F:\Mat_{k,m}(\RR)\to\CC$ be a Schwartz function. By identifying $\Mat_{k,m}(\RR)\cong\RR^{km}$, we have
$$\sum_{T\in\Mat_{k,m}(\ZZ)}F(T)=\sum_{S\in\Mat_{k,m}(\ZZ)}\hat F(S)$$
with $\hat F$ the Fourier transform of $F$ over $\RR^{km}$. The convergence of the sums is ensured for Schwartz functions.

Recall that a function $F:\Mat_{k,m}(\RR)\to\CC$ is \emph{radial} if there is a \emph{profile function} $f:\bar\Omega_m\to\CC$ such that $F(T)=f(\pi T^tT)$. It turns out that if $F$ is a radial function, so is its Fourier transform $\hat F$. Herz answered the question of the link between $\hat F$ and $f$.

\begin{theorem}[{Hankel transform, \cite[Theorem 3.4]{Her55}}]\label{thm Hankel transform}
Let $p=\frac{m+1}2$. Let $F:\Mat_{k,m}(\RR)\to\CC$ be a radial function with profile $f:\bar\Omega_m\to\CC$. Suppose that $F\in L^2(\Mat_{k,m}(\RR))$. Then $\hat F$ is also radial. Moreover the function $\det(M)^{k/2-p}f(M)\in L^2(\Mat_m(\RR))$ and the profile of $\hat F$ is given by
$$\tilde f(N):=\int_{\Omega_m}A_{k/2-p}^{(m)}(NM)\det(M)^{k/2-p}f(M)dM.$$
We call $\tilde f$ the \emph{Hankel transform} of $f$.
\end{theorem}

Recall that $A_k^{(m)}(M)$ is the matrix Bessel function defined in Equation \eqref{eq definition A_k}. The integral is well-defined for $\Re(k)>p-1$ and has an analytic continuation to $k\in\CC$. It is also entire in $M$. Finally, it is bounded for $M>0$ and invariant by conjugation of $M$. If $m=1$, then we have
$$J_k(x)=A_k^{(1)}(x^2/4)(x/2)^k.$$

Grouping the terms with respect to the profile of $F$, we deduce the following version of Poisson summation formula:
$$\sum_{\substack{M\in\Mat_m(\ZZ)\\M\geq0}}r_k(M)f(M)=\sum_{\substack{N\in\Mat_m(\ZZ)\\N\geq0}}r_k(N)\tilde f(\pi N),$$
where $r_k(M)$ is
\begin{align}\label{eq definition r_k(M)}
r_k(M):=\{T\in\Mat_{k,m}(\ZZ)\mid T^tT=M\}.
\end{align}

We will need the following Hankel transform, which is a generalization of Sonine's formula:
\begin{align}\label{eq Sonine formula}
A_{k+p}(NR)\det(R)^{k+p}=\Pi_m(0)^{-1}\int_0^RA_k(NM)\det(M)^kdM.
\end{align}
It holds for $\Re(k)>-1$, $R\geq0$ and $N$ an arbitrary complex matrix. This is proved by specializing Equation (2.6) in \cite{Her55} and recalling that $A_k(0)=\Pi_m(k)^{-1}$.

\begin{remark}
Equation (2.6) in \cite{Her55} in combination to the Poisson summation formula also allows us to prove the following formula stated in the introduction of Herz's article:
\begin{proposition}[{\cite[p.476]{Her55}}]
Let $R>0$ be a fixed matrix and $\delta>(k+m-1)/2$.
\begin{align}\label{eq Poisson summation}
\Pi_m(\delta)^{-1}\sum_{\substack{M\in\bar\Omega_m(\ZZ)\\\pi M\leq R}}r_k(M)\det(R-\pi M)^\delta=\sum_{N\in\bar\Omega_m(\ZZ)}r_k(N)A_{k/2+\delta}(\pi NR)\det(R)^{k/2+\delta}.
\end{align}
\end{proposition}
\end{remark}

\subsection{Asymptotic for Bessel functions of matrix arguments}\label{sec Bessel functions}

The goal of this section is to state the bound for the function $A_k^{(m)}$ and explain where it comes from.

\begin{theorem}[\cite{FT87,Kal01}]\label{thm asymptotic formula Kalliterakis}
Let $k\geq m$ be positive integers and $X\in\bar\Omega_m$ with eigenvalues $\lambda_1\geq\dots\geq\lambda_m\geq0$. Then
\begin{align}\label{eq asymptotic formula Kalliterakis}
A_{k/2-(m+1)/2}(X)\ll_{k,m}\sum_{\epsilon\in\{\pm1\}^m}\prod_{1\leq i<j\leq m}\left(1+\lambda_i^{1/2}+\epsilon_i\epsilon_j\lambda_j^{1/2}\right)^{-1/2}\prod_{i=1}^m(1+\lambda_i^{1/2})^{m/2-k/2}.
\end{align}
\end{theorem}

Faraut and Travaglini proved an asymptotic formula in the case of $X\in\Omega_m$ with pairwise different eigenvalues. This restriction ensures that the Hessian appearing in the method of stationary phase is non-singular. Kalliterakis worked out then the degenerated cases. Both articles are in the setting of euclidian Jordan algebras, which generalizes the sets of symmetric matrices over real numbers or hermitian matrices over complex numbers, quaternions and octonions together with the binary operation $X\circ Y=\frac{XY+YX}2$. To learn more about this subject, we recommend the book of Faraut and Korányi \cite{FK94}. Note that the normalization of the matrix Bessel function in Chapter XV differs slightly.

Recall that the \emph{pseudo-determinant} of a matrix $M$, written $\pdet(M)$, is the product of its non-zero eigenvalues. The following is easily deduced from Theorem \ref{thm asymptotic formula Kalliterakis}.
\begin{corollary}
Let $k\geq m$ be positive integers and $X\in\bar\Omega_m$. Then
$$A_k(X)\ll_{k,m}\pdet(X)^{-k/2-1/4}.$$
\end{corollary}

\begin{proof}
The first product in Equation \eqref{eq asymptotic formula Kalliterakis} is always at most 1 so it can be removed. For the second term, we have
$$\prod_{i=1}^m(1+\lambda_i^{1/2})\geq\prod_{\lambda_i\neq0}\lambda_i^{1/2}=\pdet(X)^{1/2}.$$
\end{proof}

\begin{example}
If $m=1$, we recover
$$J_k(x)=A_k(x^2/4)(x/2)^k\ll_k(1+x^2)^{-k/2-1/4}x^k\ll_k\min\{x^k,x^{-1/2}\}.$$
If $m=2$, we have by Equation (7.2') of \cite{Her55}
\begin{align*}
A_k^{(2)}(X)&=\frac1\pi\int_0^1A_k^{(1)}(x_1t)A_k^{(2)}(x_2t)t^k(1-t)^{-1/2}dt\\
	&=\frac1\pi(x_1x_2)^{-k/2}\int_0^1J_k(2\sqrt{x_1t})J_k(2\sqrt{x_2t})(1-t)^{-1/2}dt\\
	&\ll_k(x_1x_2)^{-k/2}\min\{x_1^{k/2},x_1^{-1/4}\}\min\{x_2^{k/2},x_2^{-1/4}\}
\end{align*}
where $x_1,x_2>0$ are the eigenvalues of $X$. It is not hard to see that this is the same bound as the one from the Corollary:
$$A_k^{(2)}(X)\ll_k(1+x_1)^{-k/2-1/4}(1+x_2)^{-k/2-1/4}.$$
Equation \eqref{eq asymptotic formula Kalliterakis} gives an extra win when $\abs{x_1\pm x_2}$ is large.
\end{example}

\section{Preliminary results}\label{sec preliminaries}

\subsection{Bounding \texorpdfstring{$r_k(M)$}{rₖ(N)}}

The goal of this section is to give a bound on $r_k(M)$ (defined in Equation \eqref{eq definition r_k(M)} depending on the pseudo-determinant of $M$.

\begin{lemma}\label{lem congruence of symmetric matrix}
Let $M\in\Mat_m(\ZZ)$ be a symmetric matrix of rank $r\leq m$ and $VMU$ be its Smith normal form with $U,V\in\GL_m(\ZZ)$. Then
$$U^tMU=\pmatrixtwo{N_1}{}{}{0}$$
with $N_1\in\Mat_r(\ZZ)$ of rank $r$. Moreover, we have
$$\det(N_1)\leq\pdet(M).$$
\end{lemma}

\begin{proof}
Let $VMU=D$ be the Smith normal form of $M$, with $D=\diag(d_1,\dots,d_r,0,\dots,0)$ and $d_1,\dots,d_r$ positive integers. Then
$$N:=U^tMU=U^tV^{-1}D=\pmatrixtwo{N_1}0{N_2}0.$$
Since $N$ is symmetric, we have $N_2=0$.

Clearly, the pseudo-determinant of a matrix $A$ is equal to the elementary symmetric polynomial of degree $r$ evaluated at the eigenvalues of $A$. Recall that the latter is equal to the sum of all the $r$ by $r$ principal minors of $A$, i.e.
$$\pdet(A)=\sum_{\abs I=r}A_{I,I},$$
where the sum runs over index sets $I\subseteq\{1,\dots,m\}$ of size $r$ and $A_{I,I}$ is the determinant of the corresponding submatrix. Recall also the Cauchy-Binet formula:
$$(AB)_{I,K}=\sum_{\abs K=r}A_{I,K}B_{K,J},$$
for index sets $I,J\subseteq\{1,\dots,m\}$ of size $r$. In our case, we obtain
$$\pdet(M)=\sum_{\abs I=r}(U^{-t}NU^{-1})_{I,I}=\sum_{\abs I,\abs J,\abs K=r}(U^{-t})_{I,J}N_{J,K}(U^{-1})_{K,I}.$$
On the right-hand side, $N_{J,K}$ vanishes unless $J=K=\{1,\dots,r\}$. In that case, $N_{J,K}=\det(N_1)$. We see that $\det(N_1)\divides\pdet(M)$. This concludes the proof.
\end{proof}

\begin{lemma}\label{lem reduction of r_k(M)}
Let $M\geq0$ be an integral matrix and $U\in\GL_m(\ZZ)$. Then
$$r_k(M)=r_k(U^tMU).$$
\end{lemma}

\begin{proof}
For $T\in\Mat_{k,m}(\ZZ)$, we have
$$T^tT=M\Leftrightarrow (TU)^t(TU)=U^tMU.$$
Since $U$ has an integral inverse, we conclude that $r_k(M)=r_k(U^tMU)$. This can also be deduced by the invariance of Fourier coefficients of theta series.
\end{proof}

\begin{lemma}\label{lem bound r_k(M) for non-singular}
Let $k\geq m$ and $0<M\in\Mat_n(\ZZ)$. We have
$$r_k(M)\ll_{m,k}\det(M)^{k/2}.$$
\end{lemma}

\begin{proof}
Note that by Lemma \ref{lem reduction of r_k(M)}, we can suppose that $M$ is reduced. Consider the theta series $\Theta:\Mat_m(\RR)\to\CC$ given by
$$\Theta(Z)=\sum_{T\in\Mat_{k,m}(\ZZ)}e(TZT^t)=\sum_{M\geq0}r_k(M)e(MZ).$$
Recall that $\Theta(Z)$ is a Siegel modular form of weight $k/2$ and degree $m$. If $k$ is even, this is Lemma 1 in Section 12 of \cite{Kli90}. The proof can be adapted to this case easily using that $\Theta(Z)^2$ is a Siegel modular form of integral weight. Klingen proves that
$$\abs{\Theta(Z)}^2\ll(1+\tr(Y))^{mk}\det(Y)^{-k}.$$
for an implicit constant independent of $Z$. Set $Y=M^{-1}$. Note that $\tr(Y)\ll1$ since $M$ is reduced. The function $\Theta$ is clearly 1-periodic, i.e. $\Theta(Z+A)=\Theta(Z)$ for all $A\in\Mat_k(\ZZ)$. Therefore its Fourier coefficient at $M$ is given by
$$r_k(M)=\int_{\Mat_m(\RR/\ZZ)}\Theta(X+iY)e(-MX)dX\ll\det(M)^{k/2}.$$
\end{proof}

\begin{proposition}\label{pro bound r_k(M)}
Let $k\geq m$ and $0\neq M\in\Mat_n(\ZZ)$ with $M\geq0$. We have
$$r_k(M)\ll\pdet(M)^{k/2}.$$
\end{proposition}

\begin{proof}
If $M$ has full rank $m$, then this is Lemma \ref{lem bound r_k(M) for non-singular}. If $M$ has rank $r<m$, by Lemma \ref{lem congruence of symmetric matrix}, we have
$$N=U^tMU=\pmatrixtwo{N_1}{}{}0,$$
for some matrix $U\in\GL_m(\ZZ)$. By Lemma \ref{lem reduction of r_k(M)}, $r_k(M)=r_k(N)$. Let $S\in\Mat_{k,m}(\ZZ)$ be such that $S^tS=N$. Then for $j>r$, we have
$$0=n_{jj}=\sum_{i=1}^ns_{ij}^2.$$
Therefore the $m-r$ last columns of $S$ are 0 and $r_k(N)=r_k(N_1)$. By Lemma \ref{lem congruence of symmetric matrix} and \ref{lem bound r_k(M) for non-singular}, we conclude that
$$r_k(M)=r_k(N_1)\ll\det(N_1)^{k/2}\ll\pdet(M)^{k/2}.$$
\end{proof}

We will also need a bound on the number of integral symetric matrices with fixed rank and norm.
\begin{theorem}[\cite{EK95,Sch95a}]\label{thm number of symmetric matrices}
Let $m\geq1$ and $\epsilon>0$. The number $N(t,m,r)$ of integral symmetric matrices $M\in\Mat_m(\ZZ)$ of rank $r$ and maximum norm $\norm M_\infty=t$ is
$$N(t,m,r)=O_{m,\epsilon}\left(t^{d(r)}\right)$$
where
$$d(r)=\begin{cases}\frac{m-1}2, &r=1,\\\frac{mr}2+\epsilon,&2\leq r\leq m-1,\\\frac{m(m+1)}2-1,&r=m,\end{cases}$$
\end{theorem}

\begin{proof}
The cases $2\leq r\leq m-1$ for $m\geq3$ are proven in \cite{EK95} and \cite{Sch95a}. They actually have an asymptotic for the number of matrices with norm bounded by $t$. It is clear that we could replace $t^\epsilon$ by a power of $\log$ in the estimate. The case $r=m$ is trivial.

For the case $r=1$, we can write any rational symmetric matrix $M$ of rank 1 as $uv^t$ with $u,v\in\QQ^m$ since the columns of $M$ are proportional to each other. Using that $M$ is symmetric and integral, we deduce that $M=cvv^t$ with $c\in\QQ$ and $v\in\ZZ^m$ a primitive vector. Write $c=\frac ab$ with $(a,b)=1$. We see that $b\divides v_i^2$ for all $i=1,\dots,n$. Since $v$ is primitive, this implies that $b=1$. Therefore we can write $M=cvv^t$ with $c\in\ZZ$ and $v\in\ZZ^m$ primitive.

Clearly $\norm M_\infty$ is attained on the diagonal, so there is $i_0$ such that $cv_{i_0}^2=t$. In particular, $\frac tc$ is a square. Write $t=t_1t_2^2$ and $c=c_1c_2^2$ with $t_1,c_1$ squarefree. Since $\frac tc$ is a square, we have $t_1=c_1$ and $c_2\divides t_2$. For $i\neq i_0$, we have $v_i^2\leq\frac tc=\frac{t_2^2}{c_2^2}$. Therefore
$$N(t,m,1)\ll_m\sum_{c_2\divides t_2}\left(\frac{t_2}{c_2}\right)^{m-1}.$$
If $m\geq3$, the sum over $c_2$ is always bounded. Since $t_2\leq\sqrt t$, we conclude. For $m=1$ the result is trivial and coherent with the case $r=m$. If $m=2$, we need to be more precise. Without loss of generality, we suppose that $cv_1^2=t$. Then $v_2$ is a number coprime to $v_1=\frac{t_2}{c_2}$ and bounded by the same number. There are $O(\phi(\frac{t_2}{c_2}))$ possibilities for $v_1$. Writing $d=\frac{t_2}{c_2}$, we get
$$N(t,2,1)\ll\sum_{d\divides t_2}\phi(d)=t_2.$$
The last equality is a classical identity for Euler's totient function. Again since $t_2\leq\sqrt t$, we conclude.
\end{proof}

\subsection{Bounding the convolution}
Let $G:\Mat_{k,m}(\RR)\to\RR_{\geq0}$ be a radial Schwartz function supported on the $T\in M_{k,m}(\RR)$ such that $\pi T^tT\leq I_m$ and with integral equal to 1. Consider the following functions on $\Mat_{k,m}(\RR)$:
\begin{align*}
F(T;R)&:=\onebb_{\pi T^tT\leq R},\\
G(T;a^2R)&:=\det(a^2R)^{-k/2}G(T(a^2R)^{-1/2}),\\
H(S;R,a)=(F(\cdot;R)\ast G(\cdot;a^2R))(S)&:=\int_{M_{k,m}(\RR)}F(T;R)G(S-T;a^2R)dT
\end{align*}
and their profile $f(M;R)$, $g(M;a^2R)$, $h(M;R,a)$. Note that
$$g(\pi T^tT;a^2R)=G(T;a^2R)=\det(a^2R)^{-k/2}G(T(a^2R)^{-1/2})=\det(a^2R)^{-k/2}g(\pi a^{-2}R^{-1/2}T^tTR^{-1/2})$$
and $\int_{\Mat_{k,m}(\RR)}G(T;a^2R)dT=1$. Moreover
$$\hat G(S;a^2R)=\hat G(aSR^{1/2})\quad\text{and}\quad\hat H(S;R,a)=\hat F(S;R)\hat G(aSR^{1/2}).$$

\begin{lemma}\label{lem bounds convolution}
Let $a>0$ be such that $1-a-a^2>0$. Then
$$F\left(S;\frac{1-a-a^2}{1+a}R\right)\leq H(S;R,a)\leq F(S;(1+a)^2R).$$
\end{lemma}

\begin{proof}
We have
\begin{align*}
H(S;R,a)&=\int_{\Mat_{k,m}(\RR)}F(T;R)G(S-T;a^2R)dT\\
	&=\int_{\Mat_{k,m}(\RR)}f(\pi T^tT;R)g(\pi(S-T)^t(S-T);a^2R)dT.
\end{align*}
Let $E=S-T$. Recall that $G(E;a^2R)$ is supported on the $E$ with $\pi E^tE\leq a^2R$ and $F(T;R)$ is supported on the $T$ with $\pi T^tT\leq R$. Note that for such $E$ and $T$, we have
$$0\leq(a^{-1/2}E-a^{1/2}T)^t(a^{-1/2}E-a^{1/2}T)=a^{-1}E^tE+aT^tT-E^tT-T^tE\leq2\pi^{-1}aR-E^tT-T^tE.$$
Then under the same conditions, we have
$$S^tS=(E+T)^t(E+T)=E^tE+T^tT+E^tT+T^tE\leq T^tT+\pi^{-1}(a^2+2a)R.$$
In conclusion, if $\pi T^tT\leq R$ and $\pi E^tE\leq a^2R$ with $E=S-T$, then $\pi S^tS\leq(1+a)^2R$. In other words
\begin{align*}
H(S;R,a)&=\int_{\Mat_{k,m}(\RR)}f(\pi T^tT;R)g(\pi(S-T)^t(S-T);a^2R)dT\\
	&\leq f(\pi S^tS;(1+a)^2R)\int_{\Mat_{k,m}(\RR)}G(S-T;a^2R)dT\\
	&=f(\pi S^tS;(1+a)^2R).
\end{align*}
The last equality comes from the fact that $G(T;a^2R)$ is $L^1$-normalized.

The reverse inequality is similar. The matrix $F=T-S$ also satisfies $\pi F^tF\leq a^2R$. Therefore
$$0\leq(a^{-1/2}F-a^{1/2}S)^t(a^{-1/2}F-a^{1/2}S)\leq\pi^{-1}aR+aS^tS-F^tS-S^tF$$
and
$$T^tT=(F+S)^t(F+S)\leq(1+a)S^tS+\pi^{-1}(a+a^2)R.$$
In other words, if $\pi S^tS\leq\frac{1-a-a^2}{1+a}R$ and $\pi F^tF\leq a^2R$ with $F=T-S$, then $\pi T^tT\leq R$ and
\begin{align*}
H(S;R,a)&=\int_{\Mat_{k,m}(\RR)}f(\pi T^tT;R)g(\pi(S-T)^t(S-T);a^2R)dT\\
	&\geq f\left(\pi S^tS;\frac{1-a-a^2}{1+a}R\right)\int_{\Mat_{k,m}(\RR)}G(S-T;a^2R)dT\\
	&=f\left(\pi S^tS;\frac{1-a-a^2}{1+a}R\right).
\end{align*}
\end{proof}

\section{Proof of Theorem \ref{thm R_k(R) asymptotic}}\label{sec proof of thm}

By Poisson summation formula, we have
$$\sum_{M\geq0}r_k(M)h(M;R,a)=\sum_{N\geq0}r_k(N)\tilde h(N;R,a)=\sum_{S\in\Mat_{k,m}(\ZZ)}\hat H(S;R,a).$$
The Fourier transform of a convolution is a product:
$$\hat H(S;R,a)=\hat F(S;R)\hat G(aSR^{1/2}).$$
By Theorem \ref{thm Hankel transform} and Equation \eqref{eq Sonine formula}, the profile of $\hat F$ is
$$\tilde f(N;R)=\int_0^RA_{k/2-p}(NM)\det(M)^{k/2-p}dM=\Pi_m(0)A_{k/2}(NR)\det(R)^{k/2}.$$
Since $G$ is Schwartz, so is $\hat G$. Let $N\geq0$ with $N\neq0$ and $S$ such that $\pi S^tS=N$. In particular, $S\neq0$. Then for all $s\geq1$, we have
$$\tilde g(N)=\hat G(S)\ll_s(1+\norm{S^tS})^{-2s}\ll(1+\norm N)^{-s}.$$
Note that
$$\norm{R^{1/2}NR^{1/2}}=\norm{N^{1/2}RN^{1/2}}\geq\norm{N^{1/2}\lambda_{\min}(R)I_mN^{1/2}}\geq\lambda_{\min}(R)\norm N.$$
Therefore, we have
$$\tilde g(a^2R^{1/2}NR^{1/2})\ll_s\left(1+a^2\norm{R^{1/2}NR^{1/2}}\right)^{-s}\ll_s\left(1+a^2\lambda_{\min}(R)\norm N\right)^{-s}.$$
We will fix later $s=s(r)$ and $a=a(r)$ depending on the rank $r$ of $N$.

Isolating the term $N=0$ in the Poisson summation formula, we get
$$\sum_{M\geq0}r_k(M)h(M;R,a)=\frac{\Pi_m(0)}{\Pi_m(k/2)}\det(R)^{k/2}+\sum_{\substack{N\geq0\\N\neq0}}r_k(N)\tilde h(N;R,a).$$
The rest of the proof consists in bounding the remaining sum. For this, we group the matrices $N$ by rank $r$. Note that
$$\pdet(NR)=\pdet(R[N^{1/2}])\gg\pdet((\lambda_{\min}(R)I_m)[N^{1/2}])=\pdet(N)\lambda_{\min}(R)^r$$
using the invariance by conjugation of the pseudo-determinant. Using Proposition \ref{pro bound r_k(M)} and Theorem \ref{thm number of symmetric matrices}, we get
\begin{align*}
\sum_{\substack{N\geq0\\N\neq0}}r_k(N)\tilde h(N;R,a)&=\Pi_m(0)\det(R)^{k/2}\sum_{r=1}^m\sum_{\substack{N\geq0\\\rk(N)=r}}r_k(N)A_{k/2}(NR)\tilde g(aR^{1/2}NR^{1/2})\\
	&\ll_{k,m,s}\det(R)^{k/2}\sum_{r=1}^m\sum_{\substack{N\geq0\\\rk(N)=r}}\pdet(N)^{k/2}\pdet(NR)^{-k/4-1/4}\\
	&\quad\cdot(1+a^2\lambda_{\min}(R)\norm N)^{-s}\\
	&\ll_{k,m,s}\det(R)^{k/2}\sum_{r=1}^m\lambda_{\min}(R)^{-r(k+1)/4}\\
	&\quad\cdot\sum_{\substack{N\geq0\\\rk(N)=r}}\pdet(N)^{(k-1)/4}(1+a^2\lambda_{\min}(R)\norm N)^{-s}\\
	&\ll_{k,m,s}\det(R)^{k/2}\sum_{r=1}^m\lambda_{\min}(R)^{-r(k+1)/4}\\
	&\quad\cdot\sum_{t=1}^\infty t^{r(k-1)/4+d(r)}(1+a^2\lambda_{\min}(R)t)^{-s},
\end{align*}
with $d(r)$ the exponent coming from Theorem \ref{thm number of symmetric matrices}. Note that we used that $\norm N_\infty\asymp_m\norm N$. We split the sum over $t\geq1$ between $t\leq T$ and $t>T$ for $T=(a^2\lambda_{\min}(R))^{-1}$. For $s>r(k-1)/4+d(r)+1$, we get
\begin{align*}
\sum_{\substack{N\geq0\\N\neq0}}r_k(N)\tilde h(N;R,a)&\ll_{k,m,s}\det(R)^{k/2}\sum_{r=1}^m\lambda_{\min}(R)^{-r(k+1)/4}\\
	&\quad\cdot\left(T^{r(k-1)/4+d(r)+1}+T^{r(k-1)/4+d(r)-s+1}(a^2\lambda_{\min}(R))^{-s}\right)\\
	&=\det(R)^{k/2}\sum_{r=1}^m\lambda_{\min}(R)^{-r(k+1)/4}(a^2\lambda_{\min}(R))^{-r(k-1)/4-d(r)-1}\\
	&=\det(R)^{k/2}\sum_{r=1}^ma^{-r(k-1)/2-2d(r)-1}\lambda_{\min}(R)^{-rk/2-d(r)-1}.
\end{align*}

We proved
\begin{align*}
	&\sum_{M\geq0}r_k(M)h(M;R,a)=\\
\frac{\Pi_m(0)}{\Pi_m(k/2)}&\det(R)^{k/2}+O\left(\det(R)^{k/2}\sum_{r=1}^ma^{-r(k-1)/2-2d(r)-1}\lambda_{\min}(R)^{-rk/2-d(r)-1}\right).
\end{align*}
By Lemma \ref{lem bounds convolution}, we have
\begin{align*}
\sum_{T\in\Mat_{k,m}(\RR)}F(T;R)&\leq\sum_{T\in\Mat_{k,m}(\RR)}H\left(T;\frac{1+a}{1-a-a^2}R,a\right),\\
\sum_{T\in\Mat_{k,m}(\RR)}F(T;R)&\geq\sum_{T\in\Mat_{k,m}(\RR)}H(T;(1+a)^{-2}R,a),
\end{align*}
for $1-a-a^2>0$. Since
\begin{align*}
\det\left(\frac{1+a}{1-a-a^2}R\right)^{k/2}&=\det(R)^{k/2}+O_{k,m}(a\det(R)^{k/2}),\\
\det((1+a)^{-2}R)^{k/2}&=\det(R)^{k/2}+O_{k,m}(a\det(R)^{k/2}),
\end{align*}
we conclude that
\begin{align*}
\sum_{M\geq0}r_k(M)f(M;R)&=\frac{\Pi_m(0)}{\Pi_m(k/2)}\det(R)^{k/2}+O(a\det(R)^{k/2})\\
	&\quad +O(\det(R)^{k/2}\sum_{r=1}^ma^{-r(k-1)/2-2d(r)-1}\lambda_{\min}(R)^{-rk/2-d(r)-1}).
\end{align*}
The optimal error in a fixed rank $r$ is for
$$a=\lambda_{\min}(R)^{-\frac{rk/2+d(r)+1}{rk/2-r/2+2d(r)+2}}=\lambda_{\min}(R)^{-1+\frac{2d(r)-r+2}{kr+4d(r)-r+4}}$$
Suppose that $\lambda_{\min}(R)\geq1$. Since
$$\frac{2d(r)-r+2}{kr+4d(r)-r+4}=\frac1{\frac{r(k+1)}{2d(r)-r+2}+2},$$
we see that the worst exponent is when the ratio $\frac{2d(r)-r+2}r$ is the largest. This is the case for $r=2$, where we get
$$a=\lambda_{\min}(R)^{-1+\frac{m+\epsilon}{k+2m+1+2\epsilon}}.$$
Note that $1-a-a^2>0$ if $\lambda_{\min}(R)\geq3$ since the exponent is smaller than $-\frac12$ (actually $\lambda_{\min}(r)$ only need to be slightly larger than 2). Replacing $R$ by $\pi R$, we conclude the proof of Theorem \ref{thm R_k(R) asymptotic}.

\printbibliography
\end{document}